\documentclass[11pt]{article}
\usepackage{latexsym}
\usepackage{verbatim}
\usepackage{amsmath}
\usepackage{amsthm}
\usepackage{amssymb}
\usepackage{stmaryrd}
\usepackage{enumerate}
\usepackage[all]{xy}
\def\old#1{}

\newif\ifdeveloping
\developingtrue
\ifdeveloping
\fi
\newif\ifshortversion
\shortversiontrue

\newtheorem{theorem}{Theorem}[section]
\newtheorem{corollary}[theorem]{Corollary}
\newtheorem{lemma}[theorem]{Lemma}

\def\case{\begin{array}{c@{\quad \mbox{if} \quad} l }}

\def\EE{\mathcal{E}}

\def\PP{\mathcal{P}}
\def\DD{\mathcal{D}}
\def\sign{{\rm sign}}

\newcommand{\prtime}{{\count0=\time\divide\count0 by 60
     \count1=-\count0\multiply\count1 by 60 \advance\count1 by \time
     \the\count0:\the\count1} }

\def\myheads#1;#2;{
\pagestyle{myheadings} \markboth{{\sc\hfill
     #1\hfill\protect\makebox[0cm][r]{\rm\today; \prtime}}}
     {{\sc\protect\makebox[0cm][l]{\rm\today;\ \prtime}\hfill
     #2\hfill}} \thispagestyle{myheadings} }
     \def\old#1{}
\title{Asymptotically normal distribution of some  tree families  relevant for phylogenetics, and of partitions without singletons}
\author{
     \'Eva Czabarka\footnote{The first two and last two authors were supported in part by the HUBI  MTKD-CT-2006-042794.}\\
     {\tt czabarka@math.sc.edu}\\
     University of South Carolina, Columbia, SC 29208, USA\\
        P{\'e}ter L. Erd{\H o}s\footnote{This author was supported in part
     by the Hungarian NSF contracts Nos T37846, T34702, T37758.},\\
     {\tt elp@renyi.hu}\\
     Alfr{\'e}d R{\'e}nyi Institute, 13-15 Re{\'a}ltanoda u., 1053 Budapest, Hungary\\
      Virginia Johnson\footnote{The first and third authors were  partially supported by a grant from the University of South Carolina
Promising Investigator Research Award.}\\
      {\tt  johnsonv@math.sc.edu}\\
     University of South Carolina, Columbia, SC 29208, USA \\
     Anne Kupczok\footnote{Complete affiliation:
Center of Integrative Bioinformatics Vienna (CIBIV), Max F. Perutz Laboratories (MFPL), University of Vienna, Medical University of Vienna, University of Veterinary Medicine, Vienna; current address:
Institute of Science and Technology, Austria, Am Campus 1, 3400 Klosterneuburg, Austria.}\\
     Center for Integrative Bioinformatics Vienna (CIBIV)\\
     {\tt anne.kupczok@ist.ac.at}\\
     L\'aszl\'o A. Sz\'ekely\footnote{This author was supported in part by the
     NSF DMS contracts No. 0701111 and 1000475, the NIH NIGMS contract  3R01GM078991-03S1, and by the Alexander von Humboldt Foundation at the Rheinische
     Friedrich-Wilhelms Universit\"at, Bonn.}\\
     {\tt szekely@math.sc.edu}\\
     University of South Carolina, Columbia, SC 29208, USA
     }
     \date{}
\begin{document}
\maketitle
Mathematics Subject Classification 2010:  05A15; 05A16; 05A18; 05C05

Keywords: set partition; generating function; tree; phylogeny; asymptotic enumeration; central limit theorem; local limit theorem 

\begin{abstract} P.L. Erd\H os and L.A. Sz\'ekely  [{\sl Adv.  Appl. Math.} {\bf 10}(1989), 488--496]   gave  a bijection between  rooted semilabeled trees and set partitions.    L.H. Harper's results [{\sl Ann. Math.
Stat.} {\bf 38}(1967), 410--414] on the asymptotic normality
of the Stirling numbers of the second kind translates into 
asymptotic normality  of rooted semilabeled trees with given number of vertices, when the 
number of internal vertices varies. 
The Erd\H{o}s-Sz\'ekely bijection  specializes to a bijection between phylogenetic trees and set partitions with classes of size $\geq 2$. We consider modified Stirling numbers of the second kind 
 that enumerate partitions of a fixed set
into a given number of classes of size $\geq 2$, and 
obtain their asymptotic normality as the number of classes varies.
The Erd\H{o}s-Sz\'ekely bijection 
translates this result into the
asymptotic normality  of the number of phylogenetic trees with given number of vertices, when the 
number of leaves varies. 
We also obtain asymptotic normality of the number of phylogenetic trees with given number of
leaves and varying number of internal vertices, which make more sense to students of phylogeny.
By the Erd\H{o}s-Sz\'ekely bijection this means the asymptotic normality of the number
of partitions of $n+m$ elements into $m$ classes of size $\geq 2$, when $n$ is fixed and
$m$ varies. The proofs are adaptations of the techniques of L.H. Harper [ibid.].
We provide asymptotics for the relevant expectations and variances with error term
$O(1/n)$.
\end{abstract}

\section{Semilabeled trees and set partitions}\label{semilabeled}
P\'eter Erd\H os and L\'aszl\'o  Sz\'ekely  \cite{anti} enumerated   $F(n,k)$, the number of {\sl rooted semilabeled trees} with $k$ uniquely
labeled leaves
and $n$ non-root vertices. Such trees have a root, which may or may not have  degree one, 
and is not being counted as vertex or leaf; and have $k$ leaves. Two such trees are identical, if there
is a graph isomorphism between them that maps root to root and every leaf label to the same
leaf label. The labels of the leaves come from the set $\{1,2,\ldots, k\}$ and labels are not repeated.

Erd\H os and Sz\'ekely in \cite{anti} established a bijection between the trees counted by
$F(n,k)$ and partitions of an $n$-element set into $n-k+1$ classes, under which out-degrees
of non-root vertices and the root correspond to class sizes in the partition. 
The cited result immediately
implies that $F(n,k)=S(n,n-k+1)$, where $S(a,b)$ denotes the {\sl Stirling number of the second kind} that enumerates partitions of an $a$-element set into $b$ non-empty classes; and  that 
$\sum_k F(n,k)=\sum_i S(n,i) =B_n$, the {\sl Bell number} \cite{encycl} A000110.
Any  information available on the Stirling numbers of the second 
kind translates for information on the $F$-numbers. 
For example, 
the recurrence relation 
\begin{equation} S(n,k)=S(n-1,k-1)+kS(n-1,k) \label{Snkrec}
\end{equation}
 translates to $F(n,k)=F(n-1,k)+
(n+1-k)F(n-1,k-1)$.  
However,  phylogeneticists are not interested in 
semilabeled trees
with internal vertices of degree 2 and with root degree 1. 
We  use the  term  {\sl phylogenetic tree}  for semilabeled trees that do not fall into these 
degenerate categories.   
Let $F^{\star}(n,k)$ denote the number of phylogenetic trees with $k$ leaves and $n$ non-root vertices,
and let  $S^{\star}(n,k)$ denote 
the number
of  partitions  of an $n$-element set into $k$ classes, such that each contains at least 2 elements.
The   Erd\H os--Sz\'ekely   bijection still provides
$F^{\star}(n,k)=S^{\star}(n,n-k+1)$  and $S^{\star}(n,i)=F^{\star}(n, n-i+1)$.

Felsenstein \cite{felspaper, felsbook}, and also Foulds and Robinson \cite{arscomb}
 investigated the numbers $T_{n,m}$. $T_{n,m}$ is
the number of rooted trees with $n$ labeled leaves, $m$ unlabeled internal vertices (the root
is one of them), where the root has degree at least 2 and no other internal vertices have degree 2.
Clearly  
\begin{equation} \label{csere}
T_{n,m}=F^{\star}(n+m-1, n)=S^{\star}(n+m-1,m).
\end{equation}
 If we are interested only in evaluating
certain $T_{n,m}$ numbers, formula (\ref{csere}) would suffice. However,
the  $T_{n,m}$ notation suggests that the distributions of  $F(n,k)$ and $F^{\star}(n,k)$  for large but fixed number of vertices $n$ and varying number of leaves $k$, albeit is mathematically interesting, not relevant for phylogenetics.
The relevant distribution for phylogenetics is large but fixed number of leaves $n$, and varying
number of internal vertices, with which total number of vertices varies as well.
Let $t_n=\sum_k T_{n,k}$ denote the number of all phylogenetic trees
with $n$ labeled leaves.  This sequence is A000311 in \cite{encycl}, which is the solution
to Schroeder's fourth problem \cite{schroder}. 

This paper proves central and local limit theorems for the arrays $S^\star(n,k)$ and
$T_{n,k}$, which translate into such results for $F^\star(n,i)$ and $S(n-1+m,m)$.
We compute the expectations and variances with $O(1/n)$ error term, to support
the phylogeneticists who may use our results to approximate certain large numbers.
The technique to be used is Harper's method \cite{harper}, and we heavily exploit
far-reaching asymptotic results on Bell numbers.

\section{Harper's method}
\label{method}
Harper \cite{harper} made a very elegant proof for the asymptotic normality of the
array $S(n,k)$.
We follow the interpretation
of Canfield \cite{canfield5} and  Clark \cite{clark}, who clarified and explained  the details  of \cite{harper}, although our discussion is somewhat restrictive.
Let $A(n,j)$ be an array of non-negative real 
numbers for $j=0,1,\ldots,d_n$, and  define $A_n(x)=\sum_j A(n,j)x^j $.
Observe that $\sum_j A(n,j)=A_n(1)$.
Let $Z_n$ denote the random variable, for which the probability 
$\PP(Z_n=j)=\frac{A(n,j)}{A_n(1)}$. 
In terms of $A_n(x)$, there is a well-known \cite{clark} and easy to 
verify expression 
for the expectation and variance of $Z_n$:
\begin{equation}\label{keyformula}
\EE(Z_n)= \frac{A^{\prime}_n(1)}{A_n(1)} \hbox{\ \ \ \rm and \ \ \ } \DD^2(Z_n)
=\frac{A^{\prime}_n(1)}{A_n(1)}+
\Biggl( \frac{A^{\prime}_n(x)}{A_n(x)} \Biggl)^{\prime} \Biggl|_{x=1}. 
\end{equation} 
As $\EE(Z_n)$ and $\DD(Z_n)$ are  determined by the array $A(n,j)$, we will also
write them as $\EE(A(n,.))$ and $\DD(A(n,.))$

The array $A(n,j)$ 
is called {\sl asymptotically normal} in the sense of a {\sl central limit theorem},
 if
\begin{equation}\label{Snkcentral}
\frac{1}{A_n(1)}\sum_{j=1}^{\lfloor x_n\rfloor} A(n,j)\rightarrow \frac{1}{\sqrt{2\pi}} \int_{-\infty}^xe^{-t^2/2}dt 
\end{equation}
as $n\rightarrow \infty$   uniformly in $x$, where
\begin{equation}\label{Snknormalization}
x_n=\EE(Z_n)+x\DD(Z_n).
\end{equation}

Assume now that all the roots of the polynomial
$A_n(x)$ are  non-positive real numbers, say
$\{-y_{nk}: k=1,2,\ldots, d_n\}$. Define the independent
random variables $Y_{nk}$ by $\PP(Y_{nk}=0)=y_{nk}/(1+y_{nk})$ and
  $\PP(Y_{nk}=1)=1/(1+y_{nk})$.

Observe that  the probability generating function of
 the random variable $Z_n$ is  $A_n(x)/A_n(1)$; 
and the   probability generating function of the random variable $Y_{nk}$ is $\frac{x+y_{nk}}{1+y_{nk}}$.
Since the probability generating function of a sum of independent random variables is the product
of their probability generating functions, we have that the probability generating function of 
$\sum_k Y_{nk}$ is $\prod_{k=1}^{d_n} \frac{x+y_{nk}}{1+y_{nk}}$. However, as 
$$\prod_{k=1}^{d_n} \frac{x+y_{nk}}{1+y_{nk}}=\frac{A_n(x)}{A_n(1)},$$
we conclude that
$Z_n$ and  $\sum_k Y_{nk}$ have identical distribution.
Let $G_{nj}(x)=  \PP\bigl(\frac{Y_{nj}-\EE(Y_{nj})}{\DD(Z_n)}\le x\bigl)$ denote the cumulative
distribution function
of $ \frac{Y_{nj}-\EE(Y_{nj})}{\DD(Z_n)}$ for $j=1,\ldots, d_n$.  The Lindeberg--Feller Theorem applies
(\cite{durrett} pp. 98--101) to  the sequence $\frac{Z_n-\EE(Z_n)}{\DD_n(Z_n)}=\sum_j 
\frac{Y_{nj}-\EE(Y_{nj})}{\DD_n(Z_n)}$. The condition of the cited theorem, for all $\epsilon>0$
$$\lim_{n\rightarrow\infty} \sum_{j=1}^{d_n} \int_{|y|>\epsilon} y^2dG_{nj}(y)=0$$
follows from 
\begin{equation} \label{altvariance}
\lim_{n\rightarrow \infty} \DD(Z_n)=\infty.
\end{equation}
  Therefore, the cited theorem proves 
the normal convergence (\ref{Snkcentral}), provided (\ref{altvariance}) holds and all the
roots of the 
polynomials $A_n(x)$ have non-positive real numbers. 

A sequence $a_k$ is called {\sl unimodal}, if  first it increases, and then decreases.
An array $A(n,k)$ is called {\sl unimodal}, if for every $n$,
the sequence $a_k=A(n,k)$ is such.
 A sequence $a_k$, which is 0 for $k<t$ and $\ell<k$, with $a_t\not=0$ and  
$a_\ell\not=0$,
is called 
 {\sl strictly log-concave} (SLC) if $a_k^2-a_{k-1}a_{k+1}>0$
for $t+1\leq k\leq \ell-1$.
An array $A(n,k)$ is called 
 {\sl strictly log-concave} (SLC), if  for every fixed $n$, the sequence  $a_k=A(n,k)$
is such.  
It is well-known and easy to see that any SLC sequence is unimodal in
 the variable $k$.  \old{
However, some LC sequences may not be unimodal, like 0,1,1,0,0,1,1,0. However, LC 
sequences, which do not have 0 terms both preceded and followed by non-zero terms 
(have {\sl no internal zeroes} property) are
also unimodal.  
Dobson \cite{dobson} showed the unimodality of $S(n,k)$,  Klarner
was the first   to show
the SLC property of $S(n,k)$ (see \cite{lieb}). 
}
Using Newton's Inequality, Lieb \cite{lieb}
showed that if a polynomial $\sum_{k=1}^N C_kx^k$ has only real roots, then for
$k=2,\ldots,N-1$
\begin{equation} \label{newton}
C_k^2\geq C_{k+1}C_{k-1} \Bigl( \frac{k}{k-1}\Bigl)\Bigl( \frac{N-k+1}{N-k}\Bigl),
\end{equation}
and hence the $C_k$ sequence is SLC, and 
 showed the SLC property of $S(n,k)$ through (\ref{newton}).

E.R. Canfield \cite{canfield5} noted that for 
asymptotically normal sequences  (\ref{Snkcentral}),  the SLC property 
and $\DD(Z_n)\rightarrow \infty$
implies the following {\sl local limit
theorem}:
\begin{equation}\label{Snklocal}
 \lim_{n\rightarrow \infty} \frac{\DD(Z_n)}{A_n(1)}  A(n,\lfloor x_n\rfloor )=
 \frac{1}{\sqrt{2\pi}} e^{-x^2/2}
\end{equation}
uniformly in $x$. 
Furthermore, from the fact  that the convergence of the $A(n,j)$ numbers to the
Gaussian function is actually uniform, he concluded that the number 
$k=J_n$ maximizing $A(n,k)$ satisfies
\begin{equation} \label{Snkmax}
J_n- \EE(Z_n)=o(  \DD(Z_n) );
\end{equation}
and 
\begin{equation} \label{Snkmaxvalue}
A(n, J_n)\sim \frac{1}{\sqrt{2\pi}} \frac{A_n(1)}{\DD_n(Z_n)}.
\end{equation}

For the Stirling numbers of the second kind, $A(n,j)=S(n,j)$, $A_n(1)=B_n$, and one has
\begin{equation} \label{ExpS}
\EE(S(n,.))=\frac{B_{n+1}}{B_n}-1,
\end{equation}
  \begin{equation} \label{VarS}
\DD^2(S(n,.))=\frac{B_{n+2}}{B_n}- \Bigl(\frac{B_{n+1}}{B_n}  \Bigl)^2-1.
\end{equation}
Harper \cite{harper} showed that $\sum_k S(n,k)x^k$ has distinct nonpositive roots,
that (\ref{VarS}) goes to infinity, which is sufficient for 
the asymptotic normality of the Stirling numbers of the second kind.  Harper \cite{harper} already observed (\ref{Snklocal}) for $A(n,k)=S(n,k)$.

The SLC property of $S(n,k)$
implies the SLC property and unimodality of $F(n,k)$. 
Consequently, the $F(n,k) $ array is also asymptotically normal,   in the sense of both the
central and local  limit theorems, with $\EE(F(n,.))=n+1-\EE(S(n,.))$ and 
$\DD(F(n,.))=\DD(S(n,.))$.

\section{Asymptotics for Bell numbers}
Asymptotic formula for the Bell numbers, in terms of the solution of the  unique real solution of the equation  $r  e^r=n$, 
was obtained by Moser and Wyman \cite{moserwyman}: $B_n\sim (r+1)^{-\frac{1}{2}}
\exp[n(r+r^{-1}-1)-1](1-\frac{r^2(2r^2+7r+10)}{24n(r+1)^3})$. 
Iteration easily gives $r=r(n)=\ln n -\ln \ln n +O(1)$.
 The function $r(n)$ is also known as $LambertW(n)$. The explicit form of their result
is not convenient to obtain asymptotics for the expectation and the variance, as $r$ will
vary with $n$. Canfield and 
Harper \cite{canfieldharp}, Canfield \cite{canfield} made minor modifications on the proof 
of Moser and Wyman \cite{moserwyman} to develop an estimate for $B_{n+h}$, which 
holds uniformly for $h=O(\ln n)$, using a {\sl single} $r=r(n)$ value, as $n\rightarrow \infty$:.
\begin{eqnarray}\label{canf}
B_{n+h} &=& \frac{(n+h)!}{r^{n+h}} \frac{e^{e^r-1}}{(2\pi B)^{1/2}}\\
 &\times & \Biggl(1+\frac{P_0+hP_1+h^2P_2}{e^r} +\frac{Q_0+hQ_1+h^2Q_2+h^3Q_3+h^4Q_4}
 {e^{2r}} \nonumber \\
 &+& O\bigl(e^{-3r}\bigl)\Biggl), \nonumber
\end{eqnarray}
where $B=(r^2+r) e^r$, $P_i$ and $Q_i$ are explicitly known rational functions of $r$.
We list and use in the Maple worksheet \cite{maple1} their exact values from
Canfield  \cite{bellMoser}. 
Using those, the formula (\ref{canf}) immediately provides  
asymptotics for $\EE(S(n,.))$
and $\DD(S(n,.))$, as in \cite{bellMoser} (note that \cite{bellMoser} only claimed 
$O(r/n)$ error term in (\ref{dSnk})):
\begin{equation}\label{eSnk}
\EE(S(n,.))=\frac{n}{r} -1+ \frac{r}{2(r+1)^2}  +O\bigl(\frac{1}{n}\bigl).
\end{equation}
\begin{equation}\label{dSnk}
\DD^2(S(n,.))=\frac{n}{r(r+1)}  + \frac{r(r-1)}{2(r+1)^4}  -1 +O\bigl(\frac{1}{n}\bigl).
\end{equation}
 With symbolic calculations Salvy and Shackell \cite{salvy} obtained the following asymptotics  {\em just} in terms of $n$, with a compromise at the error term: 
\begin{eqnarray} \label{salvye}
\EE(S(n,.))&=&\frac{n}{\ln n} +\frac{n (\ln \ln n +O(1/\ln n)) }{ \ln^2 n},\\
\DD^2(S(n,.))&=&\frac{n}{\ln^2 n} +\frac{n(2 \ln \ln n -1+O(1/\ln n))  }{ \ln^3 n}. \label{salvyd}
\end{eqnarray}

\section{Phylogenetic trees and set partitions without singletons} 
 \begin{theorem} \label{nosingleton}
For the sequence $A(n,j)=S^{\star}(n,j)$ the central limit theorem (\ref{Snkcentral}) and the
local limit theorem (\ref{Snklocal}) hold with 
\begin{eqnarray}
\EE(S^\star(n,.)) &= &\frac{n}{r}-r -\frac{1}{2r}+\frac{1}{2r(r+1)^2}+O\bigl(\frac{1}{n}\bigl), 
\label{elso}\\  
\DD^2(S^\star(n,.)) &= & \frac{n}{r(r+1)}-r +1-\frac{2}{r+1}-\frac{1}{2(r+1)^2}  \nonumber \\
&-&\frac{1}{2(r+1)^3}+\frac{1}{(r+1)^4}
+{O\bigl(\frac{1}{n}\bigl)}. \label{masodik}
\end{eqnarray}
Furthermore,  the number $k = J_n$ that maximizes $S^{\star}(n, k)$ satisfies
\begin{equation} \label{j1}
J_n=\frac{n}{r} + o(\frac{\sqrt{n}}{r})
\end{equation}
and
\begin{equation} \label{j2}
S^{\star}(n,J_n)=\frac{rB_{n-1}}{\sqrt{2n\pi}}(1+o(1)).
\end{equation}
\end{theorem}
It is remarkable that making an asymptotic expansion in terms of $r$ in 
(\ref{elso}), (\ref{masodik}), after a few
terms the error reduces to $O(1/n)$, as in the case of the Bell numbers in (\ref{eSnk}),
(\ref{dSnk}). Using these asymptotic expansions we obtain 
 that $\EE(S^\star(n,.))-\EE(S(n,.))=O(r)$ and $\DD^2(S^\star(n,.))-\DD^2(S(n,.))=O(r)$.
Statement (i) below follows from these remarkably small differences.
\begin{corollary}
(i)  (\ref{salvye}) and  (\ref{salvyd})  still hold
when  $S(n,.)$ is changed to $S^\star(n,.)$. \\
(ii) $A(n,k)=F^{\star}(n,k)$ satisfies (\ref{Snkcentral}) and (\ref{Snklocal}) with
$\EE(F^{\star}(n,.)) = n+1-\EE(S^{\star}(n,.)) = n-n/r +r+1+o(1/r)$ and 
$\DD(F^{\star}(n,.)) = \DD(S^{\star}(n,.)) $.
\end{corollary}

\noindent {\bf Proof} to  Theorem~\ref{nosingleton}.
We start with some facts that we need. Set $B_n^{\star}=\sum_k S^{\star}(n,k) $, the number of all partitions of an $n$-element
set not using singleton classes \cite{encycl} A000296. 
Becker \cite{becker} observed that\footnote{Identity (\ref{difference}) can be proved by the following
bijection from the partitions  with at least one singleton class of an $n$-element set, $[n]$, to the
partitions without singleton classes of an   $n+1$-element set, $[n+1]$: build a new class from
the elements of all singletons and $n+1$.}
  \begin{equation}\label{difference}
B_n=B_{n+1}^{\star}+B_n^{\star}.
\end{equation}
From $B_i=B_i^{\star}+B_{i+1}^{\star}$ for $i=1,2,\ldots , n$, and $B^{\star}_1=0$,
we obtain $B^{\star}_{n+1}=\sum_{i=1}^n B_i (-1)^{n-i}$.
As the $B_n$ sequence is strictly increasing, we immediately obtain   
$B_t-B_{t-1}<B^{\star}_{t+1}=\sum_{i=1}^t B_i (-1)^{t-i}<B_t$ 
for $t\geq 3$, and with $t=n-h$
the   asymptotical formula
\begin{equation}\label{B*onBalt}
B_{n+1}^{\star}=B_n-B_{n-1}+\ldots +(-1)^hB_{n-h} +O(B_{n-h-1}).
\end{equation}
In the special case $h=0$,  using  (\ref{canf}), we obtain:
\begin{equation}\label{B*onB}
B_{n+1}^{\star}=B_n-O(B_{n-1})=B_n \Bigl(1-O\bigl(\frac{r}{n}\bigl)\Bigl).
\end{equation}
(It turns out, as a byproduct, that almost all set partitions contain a singleton.)
We obtain the recurrence relation 
\begin{equation}\label{S*nkrec}
S^{\star}(n,k)= (n-1)S^{\star}(n-2,k-1)  +kS^{\star}(n-1,k),
\end{equation}
 according
to the case analysis whether the $n^{th}$ element is in a doubleton class or not. 
We define the 
polynomial sequence $S_n(x)=\sum_k S^{\star}(n,k)x^k$. It is easy to see that 
  $S_1(x)=0$, $S_2(x)=x$, and for $n\geq 3$ from (\ref{S*nkrec}),
\begin{equation}\label{Spolrec}
S_n(x)=(n-1) x S_{n-2}(x) + xS^{\prime}_{n-1}(x).
\end {equation}
For the proof, first we compute   $\EE(S^{\star}(n,.))$ and    $\DD(S^\star(n,.))$ exactly 
and then asymptotically.
The central and local limit theorems hinge on $\DD(S^\star(n,.))\rightarrow \infty$.
Formulae (\ref{j1}) and (\ref{j2}) follow from
 (\ref{Snkmax})
and (\ref{Snkmaxvalue}), where  $B_n^*$ is approximated with $B_{n-1}$ by (\ref{B*onB}).
Finally, Lemma~\ref{interlace} will provide the non-positive real roots of the generating
polynomial.\\

We obtain  from (\ref{keyformula}),
using (\ref{Spolrec}) repeatedly,
\begin{eqnarray*}
\EE(S^\star(n,.))&=&
\frac{B^{\star}_{n+1}}{B^{\star}_{n}}-n \frac{B^{\star}_{n-1}}{B^{\star}_n};\\
\DD^2(S^\star(n,.))&=&
    \frac{B^{\star}_{n+2}}{B^{\star}_n}+2n\frac{B^{\star}_{n+1}B^{\star}_{n-1}}{(B^{\star}_n)^2}+n(n-1)
 \frac{B^{\star}_{n-2}}{B^{\star}_n}\\
 &- &\Bigl(    \frac{B^{\star}_{n+1}}{B^{\star}_n}  \Bigl)^2 
 -n^2\Bigl(    \frac{B^{\star}_{n-1}}{B^{\star}_n}  \Bigl)^2 -n \frac{B^{\star}_{n-1}}{B^{\star}_n}-(2n+1).
\end{eqnarray*} 
To obtain (\ref{elso}) and (\ref{masodik}),
we started with the closed forms above, used (\ref{B*onBalt}) for the $B^{\star}$ numbers,  and  
substituted the $B$ numbers with (\ref{canf}),  changed $e^{-r}$ to $r/n$.
 For details, see the Maple worksheet \cite{maple1}. 

Induction immediately gives from (\ref{Spolrec}) that for $n\ge 2$ 
\begin{equation} \label{degree}
{\rm deg}( S_n(x)) = \Bigl\lfloor \frac{n}{2}\Bigl\rfloor
\end{equation}
and the root $x=0$ has multiplicity one. Hence $S^{\prime}_n(0)>0$ for $n\geq 2$.

\begin{lemma} Apart from $x=0$, the roots of $S_{2n} (x)$ and  $S_{2n+1} (x)$ are \label{interlace}
negative real numbers and every root occurs
with multiplicity one. Furthermore, if the roots of $S_{2n}(x)$ are denoted by 
$\beta_i^{(2n)}$ in increasing order, and the roots of $S_{2n-1}(x)$, $S_{2n+1}(x)$
are denoted by $\alpha_i^{(2n-1)}$,  $\alpha_i^{(2n+1)}$, both in increasing order, then the following
interlacing properties hold:
$$\beta^{(2n)}_ 1<\alpha^{(2n-1)}_ 1  <\beta^{(2n)}_ 2< \alpha^{(2n-1)}_2  < \dots <  \beta^{(2n)}_ {n-1}< \alpha^{(2n-1)}_{n-1} =0=\beta^{(2n)}_ n,$$
 $$\beta^{(2n)}_ 1<\alpha^{(2n+1)}_ 1  <\beta^{(2n)}_ 2< \alpha^{(2n+1)}_2  < \dots < \alpha^{(2n+1)}_{n-2} < \beta^{(2n)}_ {n-1}< \alpha^{(2n+1)}_{n-1} <\beta^{(2n)}_ n=0=\alpha^{(2n+1)}_n.$$
\end{lemma} 
\begin{proof}
We will use mathematical induction on $n$. The roots of $S_2(x)=S_3(x)=x$,
$S_4(x)=3x^2+x$  (roots $\beta_1^{(4)}=-1/3$ and $\beta_2^{(4)}=0$) and 
$S_5(x)=10x^2+x$  (roots $\alpha_1^{(5)}=-1/10$ and $\alpha_2^{(5)}=0$) satisfy 
 Lemma~\ref{interlace}. The inductive step follows from the following two statements
for $k\geq 2$:
\begin{enumerate} [{\rm (i)}]
\item If the roots of $S_{2n-2}(x)$ and $S_{2n-1}(x)$ occur with multiplicity one and satisfy
 $$\beta^{(2n-2)}_ 1<\alpha^{(2n-1)}_ 1  <\beta^{(2n-2)}_ 2< \alpha^{(2n-1)}_2  < \dots < \alpha^{(2n-1)}_{n-2} < \beta^{(2n-2)}_ {n-1}=0= \alpha^{(2n-1)}_{n-1} ,$$ then the roots $\beta^{(2n)}_i$ of $S_{2n}(x)$ satisfy
  $$\beta^{(2n)}_ 1<\alpha^{(2n-1)}_ 1  <\beta^{(2n)}_ 2< \alpha^{(2n-1)}_2  < \dots <  \beta^{(2n)}_ {n-1}< \alpha^{(2n-1)}_{n-1} =0=\beta^{(2n)}_ n.$$

\item  If the roots of $S_{2n-1}(x)$ and $S_{2n}(x)$ occur with multiplicity one and satisfy
 $$\beta^{(2n)}_ 1<\alpha^{(2n-1)}_ 1  <\beta^{(2n)}_ 2< \alpha^{(2n-1)}_2  < \dots <  \beta^{(2n)}_ {n-1}< \alpha^{(2n-1)}_{n-1} =0=\beta^{(2n)}_ n$$
  then the roots $\alpha^{(2n+1)}_i$  of $S_{2n+1}(x)$ satisfy
    $$\beta^{(2n)}_ 1<\alpha^{(2n+1)}_ 1  <\beta^{(2n)}_ 2< \alpha^{(2n+1)}_2  < \dots < \alpha^{(2n+1)}_{n-2} < \beta^{(2n)}_ {n-1}< \alpha^{(2n+1)}_{n-1} <\beta^{(2n)}_ n=0=\alpha^{(2n+1)}_n.$$
\end{enumerate}
First we prove (i).
In our setting  the identity (\ref{Spolrec}) specifies to 
 \begin{equation} \label{SpolrecA}
 S_{2n}(x)/x=(2n-1)S_{2n-2}(x)+ S^{\prime}_{2n-1}(x),
 \end{equation}
where the RHS is the sum of two polynomials of degree $n-1$ and $n-2$, respectively.

Set $\alpha_0^{(2n-1)}=-\infty$.  The proof hinges on the following three claims:\\
$\bullet$ The sign of $S_{2n-2}(x)$ alternates on $\alpha_i^{(2n-1)},\alpha_{i+1}^{(2n-1)}$ for
$i=0,1,...,n-3$;\\
$\bullet$ The sign of $S^{\prime}_{2n-1}(x)$ alternates on $\alpha_i^{(2n-1)},\alpha_{i+1}^{(2n-1)}$ for
$i=1,...,n-2$; and \\
$\bullet$ $\sign(S_{2n-2}(\alpha_1^{(2n-1)}))= \sign(S^{\prime}_{2n-1}(\alpha_1^{(2n-1)}))$.\\
The first claim follows from the hypotheses.\\
The second claim follows from the fact that $S^{\prime}_{2n-1}(x)$ is a polynomial of degree $n-2$
 and it has exactly one root in every interval 
 $(\alpha_i^{(2n-1)},\alpha_{i+1}^{(2n-1)})$ for $i=1,2,...,n-2$, as it must have a root 
 between consecutive roots of $S_{2n-1}(x)$.\\
The third claim follows from the facts that\\ $\sign(S_{2n-2}(\alpha_1^{(2n-1)}))=
-\sign(S_{2n-2}(-\infty))=-(-1)^{n-1}$, as $S_{2n-2}(x)$ has a single root, $\beta_1^{(2n-2)}$,
which is less than  $\alpha_1^{(2n-1)}$; and $\sign(S^{\prime}_{2n-1}(\alpha_1^{(2n-1)}))=
\sign(S^{\prime}_{2n-1}(-\infty)))=(-1)^{n-2}$, as $S^{\prime}_{2n-1}(x)$ has no root less than
  $\alpha_1^{(2n-1)}$.

{}From the three claims and (\ref{SpolrecA}) follows that the sign of $S_{2n}(x)/x$, and hence 
of $S_{2n}(x)$,
alternates on  $\alpha_i^{(2n-1)},\alpha_{i+1}^{(2n-1)}$ for
$i=1,...,n-3$; and this fact provides the required $\beta_{i+1}^{(2n)}$ root between these numbers,
$i=1,...,n-3$.

{}From the proof of the third claim and (\ref{SpolrecA}) follows that 
 $\sign(S_{2n}(\alpha_1^{(2n-1})/  \alpha_1^{(2n-1)}   )=(-1)^n$. If we show that 
 $S_{2n}(x)/x$ has a different sign at $-\infty$, then we provided the required 
 $\beta_1^{(2n)}<\alpha_1^{(2n-1)}$ root for $S_{2n}(x)/x$, and hence for $S_{2n}(x)$.
Indeed, the degree of $S_{2n-2}(x)$ is greater than the degree of $S^{\prime}_{2n-1}(x)$, and therefore
the sign of $S_{2n}(x)/x$ at $-\infty$ is the sign of $S_{2n-2}(x)$ at $-\infty$, namely $(-1)^{n-1}$.

As $S_{2n-2}(\alpha_{n-1}^{(2n-1)})=S_{2n-2}(0)=0$, the second and the third claim,
 and (\ref{SpolrecA}) imply that $S_{2n}(x)/x$ alternates on 
 $\alpha_{n-2}^{(2n-1)},\alpha_{n-1}^{(2n-1)}$, providing the required root $\beta_{n-1}^{(2n)}$ 
between these numbers, also for $S_{2n}(x)$. Finally, the last root to find is 
 $\beta_{n}^{(2n)}=0$. 

\bigskip
\noindent Next we prove (ii).
In our setting  the identity (\ref{Spolrec}) specifies to 
\begin{equation}\label{SpolrecB}
S_{2n+1}(x)/x=2nS_{2n-1}(x)+ S^{\prime}_{2n}(x),
\end{equation}
where the RHS is the sum of two polynomials of degree $n-1$. 
The proof hinges on the following three claims:\\
$\bullet$ The sign of $S_{2n-1}(x)$ alternates on $\beta_i^{(2n)},\beta_{i+1}^{(2n)}$ for
$i=1,...,n-2$;\\
$\bullet$ The sign of $S^{\prime}_{2n}(x)$ alternates on  $\beta_i^{(2n)},\beta_{i+1}^{(2n)}$ for
$i=1,...,n-1$; and \\
$\bullet$ $\sign(S_{2n-1}(\beta_1^{(2n}))= \sign(S^{\prime}_{2n}(\beta_1^{(2n)}))$.\\
The first claim follows from the hypotheses.\\
The second claim follows from the fact that $S^{\prime}_{2n}(x)$ is a polynomial of degree $n-1$
 and it has exactly one root in every interval 
$\beta_i^{(2n)},\beta_{i+1}^{(2n)}$
 for $i=1,2,...,n-1$, as it must have a root 
 between consecutive roots of $S_{2n}(x)$.\\
 The third claim follows from the facts that\\ $\sign(S_{2n-1}(\beta_1^{(2n}))=
\sign(S_{2n-1}(-\infty))=(-1)^{n-1}$, and $\sign(S^{\prime}_{2n}(\beta_1^{(2n)}))=
\sign(S^{\prime}_{2n}(-\infty)))=(-1)^{n-1}$, as neither $S^{\prime}_{2n}(x)$ nor $S_{2n-1}(x)$
 has a root less than
  $\beta_1^{(2n)}$.

{}From the three claims and (\ref{SpolrecB}) follows that the sign of $S_{2n+1}(x)/x$, and hence 
of $S_{2n+1}(x)$,
alternates on  $\beta_i^{(2n)},\beta_{i+1}^{(2n)}$ for
$i=1,...,n-2$; and this fact provides the required $\alpha_{i}^{(2n+1)}$ root between these numbers,
$i=1,...,n-2$.

As $S_{2n-1}(\beta_{n}^{(2n)})=S_{2n-1}(0)=0$, the second and the third claim,
 and (\ref{SpolrecB}) imply that $S_{2n+1}(x)/x$ alternates on 
  $\beta_{n-1}^{(2n)},\beta_{n}^{(2n)}$,
 providing the required root $\alpha_{n-1}^{(2n+1)}$ 
between these numbers, also for $S_{2n}(x)$. Finally, the last root to find is 
 $\alpha_{n}^{(2n+1)}=0$. 
\end{proof}

\section{Phylogenetic trees and set partitions in another distribution} 

\begin{theorem} \label{another}
For the array $A(n,j)=T_{n+1,j}$,  the central limit theorem (\ref{Snkcentral}) and the
local limit theorem (\ref{Snklocal}) hold with 
$$\EE(T_{n+1,.})= \frac{1-\rho}{2\rho }n+\frac{3/4-\ln 2}{\rho}+
O(1/n) \hbox{ \ \ \ \ and \ \ \ \  }$$
$$ \DD^2(T_{n+1,.})= \frac{n}{4}\Bigl(\frac{1}{\rho^2}
- \frac{2}{\rho}-1\Bigl)+  \frac{1+4\ln 2- 8\ln^2 2}{8\rho^2}                     +O(1/n),
$$ 
where $\rho=-1+2\ln 2$.
Furthermore,  the number $k = J_n$ that maximizes $T_{n+1, k}$ satisfies
\begin{equation} \label{j3}
J_{n}=\frac{1-\rho}{2\rho  }n + o(\sqrt{n}),
\end{equation}
and
\begin{equation} \label{j4}
T_{n+1,J_{n}}=\frac{n!(1+o(1))}{ \pi\sqrt{2}n\rho^{n+\frac{1}{2}}   \sqrt{(\frac{1}{\rho^2}
- \frac{2}{\rho}-1)}} .
\end{equation}
\end{theorem}
\noindent Identity (\ref{csere}) immediately implies the following central and local limit theorems
as corollaries:
\begin{equation}
\frac{1}{t_{n+1}} \sum_{j=1}^{\lfloor x_n\rfloor} S^{\star}(n+j,j)\rightarrow \frac{1}{\sqrt{2\pi}} \int_{-\infty}^xe^{-t^2/2}dt 
\end{equation}
and 
\begin{equation}
 \lim_{n\rightarrow \infty} \frac{\DD(T_{n+1,.})}{t_{n+1}}  S^{\star}(n+\lfloor x_n\rfloor,  \lfloor x_n\rfloor )\rightarrow
 \frac{1}{\sqrt{2\pi}} e^{-x^2/2}
\end{equation}
as $n\rightarrow \infty$,   uniformly in $x$, with
$x_n=\EE(T_{n+1,.})+x\DD(T_{n+1,.}).$

\noindent {\bf Proof} to Theorem~\ref{another}. Felsenstein \cite{felspaper, felsbook}  proved the recurrence relation\footnote{The recurrence is based on a case analysis whether the $n^{th}$ leaf is to
be grafted into an edge  or to be joined to an internal vertex  of an already existing tree
with $n-1$ leaves.} 
\begin{equation}\label{felsrec}
T_{n,k}=(n+k-2)T_{n-1,k-1}+kT_{n-1,k}
\end{equation}
 for $k>1$ with the initial condition $T_{n,1}=1$
for $n>1$. 
Consider the polynomials $P_n(x)=\sum_k T_{n+1,k}x^k$. Then $P_n(1)=t_{n+1}$ and the degree
of $P_n(x)$ is $n$. Felsenstein's recurrence 
relation  (\ref{felsrec}) implies the identity 
\begin{equation} \label{Prec}P_n(x)=nxP_{n-1}(x)+(x+x^2)P^{\prime}_{n-1}(x)
\end{equation}
with initial term $P_0(x)=1$, $P_1(x)=T_{2,1}x=x$.
We have for the expectation and variance, from (\ref{keyformula}),
using (\ref{Prec}) repeatedly,
\begin{eqnarray}
\EE(T_{n+1,.})&=&
\frac{t_{n+2}}{2t_{n+1}}- \frac{n+1}{2};\label{te}\\
\DD^2(T_{n+1,.})&=&
\frac{t_{n+3}}{4t_{n+1}}-\frac{t^2_{n+2}}{4t^2_{n+1}}-\frac{t_{n+2}}{2t_{n+1}}-
\frac{n+1}{4}. \label{td}
\end{eqnarray}
Consider the following bivariate generating function for $T_{n,k}$:  
\begin{equation}
\label{bivariate}
H(x,z)=\sum_{n\geq 1} \sum_k T_{n,k}x^k\frac{z^n}{n!}=\sum_{n\geq 1}  P_{n-1}(x)  \frac{z^n}{n!},
\end{equation}
in particular, $H(1,z)=
\frac{z}{1!}+\frac{z^2}{2!}+\frac{4z^3}{3!}+\frac{26z^4}{4!}+...
$ .
Flajolet \cite{flajolet}  observed the functional equation 
\begin{equation}
\label{funceq}
H(x,z)=z+x\Bigl(e^{H(x,z)}-1-H(x,z) \Bigl),
\end{equation}
which immediately follows from the Exponential Formula,
and  obtained from this equation an expression for $H(1,z)$ in terms of the Lambert function: 
$$H(1,z)=-LambertW\Bigl(
-\frac{1}{2}e^{\frac{z-1}{2}}\Bigl)+\frac{z-1}{2}.$$
He also  observed that $H(1,z)$, the EGF of the 
$t_n$ sequence, has a  singularity at $\rho=-1+2\ln 2$, and it is the only singularity at
this radius; and furthermore, for $|z|<\rho$,  there is a singular expansion of $H(1,z)$
in terms of $\Delta=
\sqrt{1-z/\rho}$, of which the first few terms are
\begin{equation}\label{singexp}
H(1,z)=\ln 2 -\sqrt{\rho} \Delta +\Bigl(\frac{1}{6}-\frac{1}{3} \ln 2\Bigl)\Delta^2-\frac{\rho^{3/2}}{36}
\Delta^3 +O(\Delta^4).
\end{equation}
Flajolet \cite{flajolet}  used (\ref{singexp}) %
to obtain asymptotic formula for $t_n$ 
and noted that asymptotic expansion can be obtained by this method. 
Using Maple, we went further and  obtained the following asymptotic expansion:
\begin{equation}\label{tnasympfurther}
t_n\sim \frac{n!}{\sqrt{\pi} \rho^{n-\frac{1}{2}}}\Biggl( \frac{1}{2n^{3/2}}   + \frac{3}{16n^{5/2}} + 
 \frac{25}{256n^{7/2}} +  O\Bigl( \frac{1}{n^{9/2}}  \Bigl)          \Biggl).
\end{equation}
Combining (\ref{te}) and (\ref{td}) with (\ref{tnasympfurther}), one obtains the asymptotics
for the expectation and the variance in Theorem~\ref{another}.
The details are on a Maple worksheet \cite{maple2}.

\begin{lemma} \label{Pnegroots}
For $n\geq 1$, the polynomial $P_n(x)$ has $n$ distinct real roots, one of them is zero, and the
other $n-1$ roots  are in the open interval  $(-1,0)$. 
\end{lemma}
\begin{proof}
We prove the theorem with mathematical induction on $n$. The small cases above are 
easy to verify.
It is easy to see (by a different induction) that $P_1(-1)=-1$ and
from (\ref{Prec}),  $P_n(-1)=(-n)P_{n-1}(-1)$, thus
\begin{equation} \label{P(-1)}
\sign(P_n(-1))=(-1)^{n}.
\end{equation}
Using the induction hypothesis, let the roots of $P_{n}(x)$ be 
$$-1<\alpha_1<\cdots <\alpha_{n-2}<\alpha_{n-1}<\alpha_n=0.$$
By Rolle's theorem, $P^{\prime}_{n}(x)$ has a root  $\beta_i$ in $(\alpha_i, \alpha_{i+1})$ for
$i=1,2,...,n-1$. From  (\ref{Prec}), observe that
$\sign(P_{n+1}(\beta_i))=-\sign(P_n(\beta_i))$. As the sign of $P_n(x)$ must alternate
on the $\beta_i$, so does $P_{n+1}(x)$, and therefore $P_{n+1}(x)$ has a root
in $(\beta_i,\beta_{i+1})$ for $i=1,2,...,n-2$. We have to find 3 more roots: one is $x=0$, 
and we will show that the other two are in the intervals $(-1,\beta_1) $ and 
$(\beta_{n-1},0)$, respectively. 

Indeed, $\sign(P_n(x))$ differs in $-1$ and $\beta_1$, since $P_n(x)$ has a single root
$\alpha_1$ between.   Also, 
$\sign(P_{n+1}(-1))=-$
$\sign(P_n(-1))$ by (\ref{P(-1)}) and $\sign(P_{n+1}(\beta_1))=-\sign(P_n(\beta_1))$ from
our earlier observation. Hence,  $\sign(P_{n+1}(x))$ differs in $-1$ and $\beta_1$,  and therefore $P_{n+1}(x)$ has a  root in $(-1,\beta_1) $.

Observe  (\ref{Prec}) with induction implies that
for $n\ge 1$ the coefficient of $x^n$ in $P_n(x)$ is positive. On one hand,
we have that for $x<0$ but $x$ sufficently close to zero, $\sign(P_{n+1}(x))=-1$. On the other hand,
$\sign(P_{n+1}(\beta_1))=-\sign(P_{n+1}(-1))=(-1)^n$, $\sign(P_{n+1}(\beta_i))=(-1)^{n+i-1}$,
and $\sign(P_{n+1}(\beta_{n-1}))=1$. Therefore $P_{n+1}(x)$  has a root in  $(\beta_{n-1},0)$.
\end{proof}


\bibliographystyle{plain}

\begin{thebibliography}{99}

\bibitem{becker} H.D. Becker, Solution to problem E 461, {\sl Amer. Math. Monthly}, {\bf 48}(1941),
701--702. 

\bibitem{canfield5} E.R. Canfield, Central and local limit theorems for the coefficients of 
polynomials of binomial type, {\sl J. Combin. Theory} Ser.  A {\bf 23}(1977), no. 3, 275--290. 

\bibitem{canfield} E.R. Canfield, Engel's inequality for Bell numbers, {\sl J. Comb. Theory} A
{\bf 72}(1995), 184--187.

\bibitem{bellMoser} E.R. Canfield, bellMoser.pdf, 6 pages manuscript.

\bibitem{canfieldharp} E.R. Canfield, L.H. Harper, A simplified guide to large antichains in the
partition lattice, {\sl Congr. Numer.} {\bf 100}(1994), 81--88. 

\bibitem{clark} L. Clark, Central and local limit theorems for excedances by conjugacy 
class and by derangement,  {\sl Integers} {\bf  2}(2002), Paper A3, 9 pp. (electronic).

\old{
\bibitem{dobinski} G. Dobinksi, Summierung der Reihe $\sum n^m/n!$ f\"ur $m=1,2,3,\ldots,$
{\sl Grunert's Archiv} {\bf 61}(1877), 333--336. 
\bibitem{dobson} A.J. Dobson, A note on Stirling numbers of the second kind, {\sl J. Comb. Theory} {\bf 5}(1968), 212--214.
}

\bibitem{durrett} R. Durrett, {\sl Probability: Theory and Examples},  Wadsworth and Brooks/Cole,
Pacific Grove, CA, 1991.

\bibitem{anti}  P.L. Erd\H os, L.A. Sz\'ekely,
Applications of antilexicographic order I: An enumerative theory of
trees, {\sl Adv.  Appl. Math.} {\bf 10}(1989), 488--496.


\bibitem{flajolet} P. Flajolet, A problem in statistical classification theory,
{\tt http://algo.inria.fr/libraries/autocomb/schroeder-html/schroeder.html}

\bibitem{felspaper} J. Felsenstein, The number of evolutionary trees, {\sl Syst. Zool.}
{\bf 27}(1)(1978), 27--33. Corrigendum   {\sl Syst. Zool.}
{\bf 30}(1981), 122.

\bibitem{felsbook} J. Felsenstein, {\sl Inferring Phylogenies}, Sinauer Associates, Sunderland, Massachusetts, 2004.

\bibitem{arscomb} L.R. Foulds, R.W. Robinson, 
Enumeration of phylogenetic trees without points of degree two.
{\sl Ars Combin.} {\bf 17}(1984), A, 169--183. 


\bibitem{harper} L.H. Harper, Stirling behaviour is asymptotically normal, {\sl Ann. Math.
Stat.} {\bf 38}(1967), 410--414.


\bibitem{lieb} E.H. Lieb, Concavity properties and a generating function for Stirling numbers,
{\sl J. Comb. Theory} {\bf 5}(1968), 203--206.




\bibitem{moserwyman} L. Moser, M. Wyman, An asymptotic formula for the Bell numbers,
{\sl Trans. Roy. Soc. Canada} III {\bf 49}(1955), 49--53.

\bibitem{salvy} B. Salvy, J. Shackell, Asymptotics of the Stirling numbers of the second kind, {\sl Studies in Automatic Combinatorics}
II (1997). Published electronically.

\bibitem{schroder} E. Schroeder, Vier combinatorische Probleme, {\sl Z. f. Math. Phys.}
{\bf 15}(1870), 361--376.

\bibitem{encycl}  N.J.A. Sloane,  {\sl The On-Line Encyclopedia of Integer Sequences}\\
{\tt http://www.research.att.com/$\sim$njas/sequences/} 

\bibitem{maple1} {\tt http://www.math.sc.edu/$\sim$szekely/Aprilattemptformal.pdf}

\bibitem{maple2} {\tt http://www.math.sc.edu/$\sim$szekely/copykiserletformal.pdf}

\end{thebibliography}


\end{document}